\documentclass[11pt]{article}

\usepackage{amsmath,amssymb,amsthm}
\usepackage[dvips]{graphicx}
\usepackage{latexsym}
\usepackage{eucal}
\usepackage[pdftex]{hyperref}
\usepackage[latin1]{inputenc}
\usepackage[english,activeacute]{babel}

\newtheorem{theorem}{Theorem}[section]

\newtheorem{lemma}[theorem]{Lemma}

\theoremstyle{definition}

\newtheorem{remark}[theorem]{Remark}

\newtheorem{example}[theorem]{Example}

\title{{\bf\Large Problems with mean curvature-like operators and three-point boundary conditions}}
\author{{\large Dionicio Pastor Dallos Santos }\footnote{Email: dionicio@ime.usp.br}\hspace{2mm}
{\bf\large}\vspace{1mm}\\
{\small Department of Mathematics, IME-USP, Cidade Universit\'aria,}\\
{\small  CEP 05508-090, S\~ao Paulo, SP, Brazil}
}

\date{}
\begin{document}
\maketitle

\begin{abstract}
In this paper we study the existence of solutions for a new class of nonlinear differential equations with three-point boundary conditions. Existence of solutions are obtained by using the Leray-Schauder degree.
\end{abstract}

 \medskip



 \medskip
 \noindent
 The author declares that there is no conflict of interest regarding the publication of this article.


\section{Introduction}
The purpose of this article is to obtain  some existence results for nonlinear boundary value problems of the form
\begin{equation}\label{equa1}
\left\{\begin{array}{lll}
(\varphi(u' ))'   = f(t,u,u') & & \\
l(u,u')=0, 
\end{array}\right.
\end{equation}
where $l(u,u')=0$ denotes the boundary conditions  $u(0) = u'(0)= u'(T)$ \  or  \  $u(0) = u(T)= u'(T)$  on the interval $\left[0,T\right]$, $\varphi:\mathbb{R}\rightarrow (-a,a)$ is a  homeomorphism such that $\varphi(0)=0$, $f:\left[0, T\right]\times \mathbb{R} \times \mathbb{R}\rightarrow \mathbb{R}$ is a  continuous function, and  $a$ and $T$ are positive real numbers. The course, a \textsl{solution } of (\ref{equa1}) is  a function $u:\left[0, T\right]\rightarrow \mathbb{R}$ of class $C^{1}$, satisfying  the boundary conditions, such that  \ $\varphi(u')$ is continuously differentiable and   verifies $(\varphi(u'(t) ))'  = f(t,u(t),u'(t))$ for all $t\in\left[0,T\right]$.
 
Several papers have been recently devoted to  the study of nonlinear ordinary differential equations of the form (\ref{equa1}), where $l(u,u')=0$ denotes the periodic, Neumann or Dirichlet boundary conditions. In particular, for $\varphi(s)=s/\sqrt{1+s^{2}}$ and Dirichlet conditions, one can consult \cite{bon, bon1, bru, li}.

In \cite{pierluigi}, the authors have studied the problem (\ref{equa1}), where  $f : [0, T]\times \mathbb{R}^{n} \times \mathbb{R}^{n} \rightarrow \mathbb{R}^{n}$  is a Carath\'eodory function, $\varphi: \mathbb{R}^{n} \rightarrow  B_{1}(0)\subset \mathbb{R}^{n}$,  and $l(u,u')=0$ denotes the  periodic  boundary conditions. They obtained the existence of solutions by means of the Leray-Schauder degree theory. The interest in this class of
nonlinear operators $u\mapsto (\varphi(u' ))' $  is mainly due to the fact that they include  the mean curvature operator
\begin{center}
$u\mapsto $div$ \left(\frac{\nabla u}{\sqrt{1+ \left|\nabla u\right|^{2}}}\right)$.
\end{center}

In 2006, C. Bereanu and J. Mawhin \cite{ma}, using the Leray-Schauder degree theory, studied the nonlinear problems of the form
\[
\left\{\begin{array}{lll}
(\varphi(u' ))'   = f(t,u,u') & & \\
u(0)=0=u(T) & & \quad \quad 
\end{array}\right.
\] 
and
\[
\left\{\begin{array}{lll}
(\varphi(u' ))'   = f(t,u,u') & & \\
u'(0)=0=u'(T), & & \quad \quad 
\end{array}\right.
\] 
where  $f:\left[0, T\right]\times \mathbb{R} \times \mathbb{R}\rightarrow \mathbb{R}$ is a  continuous function and $\varphi:\mathbb{R}\rightarrow (-a,a)$ is a  homeomorphism such that $\varphi(0)=0$. They obtained the following existence theorems, respectively.
\begin{theorem} If the function  $f$ satisfies the condition                       
\begin{center}
$ \exists \, c>0 \text{ such that }  \left|f(t,x,y)\right|\leq c < \frac{a}{2T},$ \quad $\forall (t,x,y)\in[0, T]\times \mathbb{R}\times\mathbb{R}$,
\end{center}
the Dirichlet problem  has  at least  one solution.
\end{theorem}

\begin{theorem}
 Let $f$ be continuous. Assume that  $f$ satisfies the following conditions.
\begin{enumerate}
\item  There exists $c \in C$ such that $\left\|c^{-}\right\|_{L^{1}}< \frac{a}{2}$ and $f(t,x,y)\geq c(t)$ for all $(t,x,y)\in [0, T]\times \mathbb{R}\times \mathbb{R}$. 
\item There exist $R>0$ and $\epsilon \in \left\{-1,1\right\}$ such that
\begin{center}
$ \epsilon \int_0^T f(t,u(t),u'(t))dt >0$ \ if \ $u_{m}\geq R,  \ \ \left\|u'\right\|_{\infty}\leq M$,
\end{center}
\begin{center}
$\epsilon \int_0^T f(t,u(t),u'(t))dt <0$ \ if \ $u_{M}\leq -R,  \ \ \left\|u'\right\|_{\infty}\leq M$,
\end{center}
\end{enumerate}
where $L=\max\left\{ \left| \varphi^{-1}(2 \left\|c^{-}\right\|_{L^{1}} ) \right|, \left|\varphi^{-1}(-2\left\|c^{-}\right\|_{L^{1}} )\right| \right\}$. Then the Neumann problem has at least one solution.
\end{theorem}
Inspired by those results, we study the problems (\ref{equa1}) by  using  similar  topological methods based upon Leray-Schauder degree. The main contribution of this paper is the extension of some results above cited  to a more general type of boundary conditions.

The  paper is organized as follows. In Section 2, we establish the notation, terminology, and various lemmas which will be  used throughout this paper. Section 3  is devoted to the study of existence of solutions for (\ref{equa1}) with boundary  conditions of type $u(0) = u'(0)= u'(T)$. In Section 4, for  $u(0)=u(T)=u'(T)$ boundary  conditions, we investigate the existence of at least one solution  for (\ref{equa1}). Such problems do not seem to have been studied in the literature. In the present paper generally we follow the ideas of Bereanu and Mawhin \cite{man4, ma}.


\section{Notation and preliminaries}
\label{S:-1}
We first introduce some notation. For fixed $T$, we denote  the usual norm in $L^{1}=L^{1}(\left[0,T\right], \mathbb R)$ for $\left\| \cdot \right\|_{L^{1}}$. Let $C=C(\left[0,T\right], \mathbb R)$ denote the Banach space of   continuous functions from $\left[0, T\right]$ into $\mathbb R$, endowed witch the  uniform norm  $\left\| \cdot \right\|_{\infty}$,  $C^{1}=C^{1}(\left[0,T\right], \mathbb R)$  denote the Banach space of continuously differentiable functions  from $\left[0, T\right]$ into $\mathbb R$, equipped witch the usual norm $\left\|u\right\|_{1}=\left\|u\right\|_{\infty} +  \left\|u' \right\|_{\infty}$, and  $B_{\rho}(0)$  the open ball of center $0$ and radius $\rho$ in any normed space.

We introduce the following applications:
\medskip

\noindent 
the  \textit{Nemytskii operator} $N_f:C^{1} \rightarrow C $, 
\begin{center}
$N_f (u)(t)=f(t,u(t),u'(t))$, 
\end{center}
the  \textit{integration operator}   $H:C \rightarrow C^{1}$, 
\begin{center}
$ H(u)(t)=\int_0^t u(s)ds$, 
\end{center}
\noindent 
the following continuous linear applications:
\begin{center}
$K:C \rightarrow C^{1}, \ \  K(u)(t)=-\int_t^T u(s)ds $,
\end{center}
\begin{center}
$Q:C \rightarrow C, \ \  Q(u)(t)=\frac{1}{T}\int_0^T u(s)ds$,
\end{center}
\begin{center}
$S:C \rightarrow C, \ \  S(u)(t)=u(T)$,
\end{center}
\begin{center}
$P:C \rightarrow C, \ \  P(u)(t)=u(0)$.
\end{center}

For  $u\in C$, we write
\begin{center}
$ u_{m}=\displaystyle  \min_{[0,T]} u,  \  u_{M}= \displaystyle \max_{[0,T]}u,  \   u^{+}=\displaystyle \max\left\{ u,0 \right\},  \  u^{-}=\displaystyle \max \left\{ -u,0 \right\}$.
\end{center}

For the convenience of the reader we recall some results, which will be crucial in the proofs of our results. The following results are taken from \cite{dallos}(see also  \cite{ma1, man6}, respectively). The firs one is needed in the construction of the equivalent fixed point problem.
\begin{lemma}\label{lebema}
Let $B=\left\{h\in C:\left\|h \right\|_{\infty}<a/2\right\}$. For  each $h\in B$,  there exists a unique  $Q_{\varphi}=Q_{\varphi}(h)\in Im(h)$ (where $Im(h)$ denotes the range of  $h$) such that
\begin{center}
$\int_0^T \varphi^{-1}(h(t)-Q_{\varphi}(h))dt = 0$.
\end{center}
Moreover, the function    $Q_{\varphi}:B\rightarrow \mathbb{R}$ is continuous  and sends bounded sets into bounded sets.
\end{lemma}

The second one is an extension of the homotopy invariance property for Leray-schauder degree.
\begin{lemma}\label{propo1}
Let $X$  be a real Banach space, $V\subset [0,1]\times X$ be an open, bounded set and  $M$ be a completely continuous operator  on  $\overline{V}$  such that $x\neq M(\lambda,x)$  for each  $(\lambda,x)\in\partial V$. Then the Leray-Shauder degree
\begin{center}
$ deg_{LS}(I-M(\lambda,.),V_{\lambda},0)$
\end{center}
is well defined  and independent of   $\lambda$  in   $[0,1]$, where  $V_{\lambda}$ is the open, bounded (possibly empty) set defined  by  $V_{\lambda}=\left\{x\in X:(\lambda,x)\in V\right\}$.
\end{lemma} 


\section{Problems with  bounded homeomorphisms}
\label{S:0}

In this section we are interested  in  boundary value problems of the type
\begin{equation}\label{diri1}
\left\{\begin{array}{lll}
(\varphi(u' ))'   = f(t,u,u') & & \\
u(0)=u'(0)=u'(T), 
\end{array}\right.
\end{equation}
where  $\varphi:\mathbb{R} \rightarrow (-a,a)$ is  a homeomorphism, $\varphi(0)=0$ and $f:\left[0, T\right]\times \mathbb{R} \times \mathbb{R}\rightarrow \mathbb{R} $ is a continuous function. In order to apply Leray-Schauder degree theory  to show  the existence of  at least one  solution of (\ref{diri1}), we consider  for  $ \lambda \in [0,1]$, the  family of boundary value problems
\begin{equation}\label{diri2}
\left\{\begin{array}{lll}
(\varphi(u' ))'   = \lambda N_{f}(u)+(1-\lambda)Q(N_{f}(u)) & & \\
u(0)=u'(0)=u'(T). 
\end{array}\right.
\end{equation}
Notice that \ref{diri2} coincide, for $\lambda=1$, with (\ref{diri1}). Now, we introduce the set
\begin{center}
$\Omega = \left\{(\lambda,u)\in[0,1]\times C^{1}:\left\|\lambda H(N_{f}(u) -Q(N_{f}(u)))+ \varphi(P(u)) \right\|_{\infty}<a\right\}$, 
\end{center}
 where clearly $\Omega$ is  an open set in $[0,1]\times C^{1}$, and  is  nonempty because $\left\{0\right\}\times C^{1}\subset\Omega$. Introduce also the  operator  $M:\Omega \rightarrow C^{1}$  defined by 
\begin{equation}\label{diri3}
  M(\lambda,u)=P(u)+Q(N_{f}(u)) + H(\varphi^{-1}\left[\lambda H( N_f (u)-Q(N_{f}(u)))+\varphi(P(u))\right]). 
\end{equation}    
Here  $\varphi^{-1}$  with an abuse of notation is understood as the operator $\varphi^{-1}:B_{a}(0) \subset C \rightarrow C$ defined by $\varphi^{-1}(v)(t)=\varphi^{-1}(v(t))$. The symbol $B_{a}(0)$ denoting the open ball of center $0$ and radius $a$ in $C$. It is clear that $\varphi^{-1}$ is continuous and sends bounded sets into bounded sets.

When the boundary conditions are periodic or Neumann, an operator has been considered by Bereanu and Mawhin \cite{ma}.

The following lemma plays a pivotal role to study the solutions of the  problem (\ref{diri2}).
\begin{lemma}\label{dallos1}
The operator  $M:\Omega \rightarrow C^{1}$ is well  defined  and continuous. Moreover, if  $(\lambda,u)\in \Omega$  is such that  $M(\lambda,u)=u$, then   $u$ is solution of  (\ref{diri2}).
\end{lemma}
\begin{proof}
Let $(\lambda,u)\in \Omega$. We show that in fact $M(\lambda,u)\in C^{1}$. The continuity of $M(\lambda,u)$  is a straightforward consequence of the fact that this map is a composition of continuous maps. In addition
\begin{center}
 $\left(M(\lambda,u)\right)^{'}=\varphi^{-1}\left[\lambda H( N_f (u)-Q(N_{f}(u)))+\varphi(P(u))\right]$.
\end{center}
That is, $\left(M(\lambda,u)\right)^{'}$ is a composition of continuous operators and thus $M(\lambda,u)\in C^{1}$. The continuity of $M$  follows by the continuity of the operators which compose it $M$.
 
Now suppose that   $(\lambda,u)\in \Omega$  is such that  $M(\lambda,u)=u$. It follows from (\ref{diri3}) that
\begin{equation}\label{diri4}
 u(t)= u(0)+Q(N_{f}(u))(t) + H(\varphi^{-1}\left[\lambda H( N_f (u)-Q(N_{f}(u)))+\varphi(P(u))\right])(t) 
\end {equation}
for all  $t\in\left[0,T\right]$. Then, taking $t=0$ we get 
\begin{equation}\label{diri5} 
Q(N_{f}(u))=0.
\end {equation}
Differentiating (\ref{diri4}), we obtain that
\begin{align*}
u'(t)&= \varphi^{-1}\left[\lambda H( N_f (u)-Q(N_{f}(u)))+\varphi(P(u))\right](t)\\
&= \varphi^{-1}\left[\lambda H( N_f (u)-Q(N_{f}(u)))(t) +\varphi(u(0)) \right].
\end{align*}
In particular,
\begin{center}
$u(0)=u'(0)=u'(T)$.
\end{center}
Applying  $\varphi$  to both of its members, differentiating again and using (\ref{diri5}), we deduce that
\begin{center}
 $(\varphi(u'(t) ))'= \lambda N_{f}(u)+(1-\lambda)Q(N_{f}(u))(t)$
\end{center}
for all $t\in\left[0, T\right]$. Thus, $u$ satisfies problem (\ref{diri2}). This completes the proof.
\end{proof}
\begin {remark}
 Note that  for $\lambda \in[0,1]$, if $u$ is a solution of (\ref{diri2}), then $Q(N_{f}(u))=0$.
\end {remark}

 The following lemma gives a priori bounds for the possible fixed points of $M$.
\begin{lemma}\label{santos}
 Assume that  $f$ satisfies the following conditions.
\begin{enumerate}
\item  There exists  $M_{1}<M_{2}$  such that for all  $u\in C^{1}$,
\begin{center}
$\int_0^T f(t,u(t),u'(t))dt \neq 0$ \ if \ $u_{m}'\geq M_{2} $,
\end{center}
\begin{center}
$\int_0^T f(t,u(t),u'(t))dt \neq 0$ \ if \ $u_{M}'\leq M_{1} $.
\end{center}
\item There exists $c \in C$ such that 
\begin{center}
$f(t,x,y)\geq c(t)$ \  and \ $L+2\left\|c^{-}\right\|_{L^{1}}< a$
\end{center}
for all $(t,x,y)\in [0, T]\times \mathbb{R}\times \mathbb{R}$ and $L=\max\left\{ \left| \varphi(M_{2})\right|, \left| \varphi(M_{1})\right| \right\}$.
\end{enumerate}

If  $(\lambda,u)\in \Omega $ is such that  $u=M(\lambda,u)$, then
\begin{center}
$ \left\|\lambda H( N_f (u)-Q(N_{f}(u)))+\varphi(P(u))\right\|_{\infty}<  L + 2\left\|c^{-}\right\|_{L^{1}}$      \  and  \   $\left\|u\right\|_{1}< r(2+T)$,
\end{center}
where $$r=\max\left\{\left|\varphi^{-1}(L+2\left\|c^{-}\right\|_{L^{1}})\right|, \left|\varphi^{-1}(-L-2\left\|c^{-}\right\|_{L^{1}})\right|\right\}$$.
\end{lemma}
\begin{proof}
Let   $(\lambda, u) \in \Omega$  be such that $u=M(\lambda,u)$. Using Lemma \ref{dallos1}, $u$ is a solution of  problem (\ref{diri2}), then 
\begin{equation}\label{ya1}
Q(N_{f}(u))=0,
\end{equation}
and thus $u$ solves the problem 
\begin{center}
$(\varphi(u' ))' = \lambda N_{f}(u),  \ \  u(0)=u'(0)=u'(T)$. 
\end{center}
Hence
\begin{equation}\label{ya2}
\varphi(u'(t))=\lambda H(N_{f}(u)-Q(N_{f}(u)))(t) + \varphi(u(0))   \ \ \  \left(t\in\left[0, T\right]\right).
\end{equation}
On the other hand using hypothesis 1, we have that
\begin{center}
$u_{m}'<M_{2}$  \   and   \  $u_{M}'> M_{1}$.
\end{center}
It follows that there exists  $\omega \in [0,T]$  such that  $M_{1}<u'(\omega)<M_{2}$  \  and  
\begin{center}
$\int_\omega ^t (\varphi(u'(s)))'ds= \lambda \int_\omega ^t  N_{f}(u)(s)ds$,
\end{center}
which implies that
\begin{center}
$\left| \varphi(u'(t))\right| \leq \left| \varphi(u'(\omega))\right|+  \int_0 ^T \left|f(s,u(s),u'(s))\right|ds$, 
\end{center}
where
\begin{center}
$\int_0 ^T \left|f(s,u(s),u'(s))\right|ds\leq \int_0 ^T f(s,u(s),u'(s))ds +2\int_0 ^T c^{-}(s)ds$.
\end{center}
Hence
\begin{center}
$\left| \varphi(u'(t))\right| < L + 2\left\|c^{-}\right\|_{L^{1}}$,
\end{center}
where  $L= \max \left\{ \left| \varphi(M_{2})\right|, \left| \varphi(M_{1})\right| \right\}$ \  and  \ $t\in [0,T]$. Using the equality above (\ref{ya2}), we have 
\begin{center}
$\left\|\lambda H(N_{f}(u)-Q(N_{f}(u))) + \varphi(P(u)) \right\|_{\infty} < L + 2\left\|c^{-}\right\|_{L^{1}}<a$,
\end{center}
which provides  
\begin{center}
$\left\|u'\right\|_{\infty}< r$,
\end{center}
where  $r= \max\left\{\left|\varphi^{-1}(L+2\left\|c^{-}\right\|_{L^{1}})\right|,\ \left|\varphi^{-1}(-L-2\left\|c^{-}\right\|_{L^{1}})\right|\right\}$.
Because  $u\in C^{1}$ is such that  $u'(0)=u(0)$ we have that
\begin{center}
$\left|u(t)\right|\leq \left|u(0)\right|+ \int_ 0 ^T  |u'(s)| ds < r +rT   \  \  (t\in[0,T] )$.
\end{center}
 So, we obtain that  $\left\|u\right\|_{1} =\left\|u\right\|_{\infty}+ \left\|u'\right\|_{\infty}< r+rT +r=r(2+T)$. This completes the proof of Lemma \ref{santos}.
\end{proof}

Let $\rho, \  \kappa \in \mathbb{R}$ be such that  $L+2\left\|c^{-}\right\|_{L^{1}} <\kappa<a,  \  \rho>r(2+T)$ and consider the set 
\begin{center}
$V=\left\{(\lambda,u)\in[0,1]\times C^{1}:\left\|\lambda H(N_{f}(u)-Q(N_{f}(u))) + \varphi(P(u)) \right\|_{\infty}<\kappa, \  \left\|u\right\|_{1}<\rho\right\}$.
\end{center}
Since  the set  $\left\{0\right\}\times\left\{u\in C^{1}:\left\|u\right\|_{1}<\rho, \   \left\| \varphi(P(u))\right\|_{\infty}< \kappa \right\}\subset V$, then we deduce that $V$ is nonempty. Moreover, it is clear that $V$ is open and bounded in  $ [0,1]\times C^{1}$  and  $\overline{V}\subset\Omega$. On the other hand using an argument similar to the one introduced in the proof of Lemma \ref{dallos1}, it is not difficult to see that $M:\overline{V}\rightarrow C^{1}$  is well defined and continuous. Furthermore, using Lemma \ref{santos}, we have that
\begin{center}
$ u\neq M(\lambda,u)$  \  for all  \ $(\lambda,u) \in\partial V$.
\end{center}
\begin{lemma}\label{santos11}
The operator $M:\overline{V}\rightarrow C^{1}$ is completely continuous.
\end{lemma}
\begin{proof}
Let $\Lambda \subset \overline{V}$ be a bounded set. Then, if $(\lambda, u)\in\Lambda$, there exists a constant $\delta>0$ such that  
\begin{equation}\label{delta}
\left\|(\lambda, u)\right\|=  \max\left\{\left\|\lambda \right\|, \  \left\|u\right\|_{1}\right\}\leq \delta.
\end{equation}
Let us show that  $\overline{M(\Lambda)}\subset  C^{1}$ is compact. To see this consider first a sequence $(v_{n})_{n}$ of  $M(\Lambda)$  and let 
 $(\lambda_{n}, u_{n})_{n}$  be a sequence in $\Lambda$  such that $v_{n}=M(\lambda_{n}, u_{n})$. Using (\ref{delta}), we have that there exists a constant $L_{1}>0$ such that, for all $n\in \mathbb{N}$,
\begin{center}
$\left\|N_{f}(u_{n})\right\|_{\infty}\leq L_{1}$.
\end{center}
 
 Because  $\left\|\lambda_{n}H(N_{f}(u_{n})-Q(N_{f}(u_{n}))) + \varphi(P(u_{n}))\right\|_{\infty} \leq \kappa <a$  for all  $n\in \mathbb{N}$, it follows that the sequence  $(\lambda_{n}H(N_{f}(u_{n})-Q(N_{f}(u_{n}))) + \varphi(P(u_{n})))_{n}$ is bounded in $C$. Moreover, for  any $t, t_{1}\in\left[0, T\right]$ and for all $n\in \mathbb{N}$ we have
\begin{align*}
&\left|\lambda_{n}H(N_{f}(u_{n})-Q(N_{f}(u_{n})))(t) + \varphi(u_{n}(0)) - \lambda_{n}H(N_{f}(u_{n})-Q(N_{f}(u_{n})))(t_{1}) - \varphi(u_{n}(0)) \right|\\
&\leq \left|\int_{t_1}^t f(s,u_{n}(s),u_{n}'(s))ds - \int_{t_1}^t Q(N_{f}(u_{n}))(s) ds\right|\\
&\leq \left|t-t_1\right|\left\| N_f(u_{n})\right\|_{\infty} +  \left|t-t_1\right|\left\| Q(N_f(u_{n}))\right\|_{\infty}\\
&\leq L_{1}\left|t-t_1\right| +  L_{1} \left|t-t_1\right|\\
&\leq 2L_{1}\left|t-t_1\right|,
\end{align*} 
which implies that   $(\lambda_{n}H(N_{f}(u_{n})-Q(N_{f}(u_{n}))) + \varphi(P(u_{n})))_{n}$ is equicontinuous. Thus, by the Arzel\`a-Ascoli  theorem  there is a subsequence of $(\lambda_{n}H(N_{f}(u_{n})-Q(N_{f}(u_{n}))) + \varphi(P(u_{n})))_{n}$, which we call  $(\lambda_{n}H(N_{f}(u_{j})-Q(N_{f}(u_{j}))) + \varphi(P(u_{j})))_{j}$, which is  convergent in $C$. Using  that  $\varphi^{-1}: B_{a}(0)\subset C \rightarrow  C$ is continuous it follows from   
\begin{center}
$(M(\lambda_{n_{j}},u_{n_{j}}))'=\varphi^{-1} \left[\lambda_{n}H(N_{f}(u_{j})-Q(N_{f}(u_{j}))) + \varphi(P(u_{j})) \right] $
\end{center}

 that the sequence    $((M(\lambda_{n_{j}},u_{n_{j}}))')_{j}$  is  convergent in $C$. Then, passing to a subsequence if necessary, we obtain 
 that   $(v_{n_{j}})_{j}= (M(\lambda_{n_{j}},u_{n_{j}}))_{j}$ is  convergent in  $C^{1}$. Finally, let  $(v_{n})_{n}$  be a sequence in  $\overline{M(\Lambda)}$. Let  $(z_{n})_{n}\subseteq M(\Lambda)$  be such that
\[
\lim_{n \to \infty}\left\| z_{n}-v_{n}\right\|_{1}=0.
\]
Let  in addition $(z_{n_{j}})_{j}$ be a subsequence of  $(z_{n})_{n}$   that  converges to $z$. Therefore, $z\in \overline{M(\Lambda)}$ and  $(v_{n_{j}})_{j}$ converge  to $z$. This concludes the proof.
\end{proof}


\subsection{Main result}
\label{S:0}

In this subsection, we present and prove an existence theorem for (\ref{diri1}). We denote by $deg_{B}$ the Brouwer degree and  for $deg_{LS}$ the 
Leray-Schauder degree, and define the mapping $G:\mathbb{R}^{2}\rightarrow \mathbb{R}^{2}$ by 
\begin{center}
$G:\mathbb{R}^{2}\rightarrow \mathbb{R}^{2},  \   (x,y)\mapsto (-\frac{1}{T}\int_0^T f(t,x+yt,y)dt, y-x)$.
\end{center} 

\begin{theorem}\label{prin1}
Let $f:\left[0,T\right]\times \mathbb{R}\times \mathbb{R}\rightarrow \mathbb{R}$ be continuous and satisfy condition (1) and (2) of  Lemma \ref{santos}. Then, for all $\rho>r(2+T)$ and  for all $\kappa\in \mathbb{R}$ with  $L+2\left\|c^{-}\right\|_{L^{1}} <\kappa<a$, we have 
\begin{center}
$deg_{LS}(I-M(1,\cdot), V_{1},0) = deg_{B}(G, \Delta, 0)$,
\end{center}
where $V_{1}$ is the open, bounded set defined by $V_{1}=\left\{u\in C^{1}:(1,u)\in V\right\}$ and $\Delta=B_{\rho}(0) \cap  \mathbb{R}^{2} \cap \left\{(x,y)\in \mathbb{R}^{2}:\left|\varphi(x)\right|<\kappa \right\}$. If furthermore
\begin{center}
$deg_{B}(G, \Delta, 0)\neq 0$,
\end{center}
problem (\ref{diri1}) has at least one solution.
\end{theorem} 

\begin{proof} 
  Let $M$ be the operator given by (\ref{diri3}). Using Lemma \ref{propo1}, we deduce that
\begin{center}
$deg_{LS}(I-M(0,.),V_{0},0)=deg_{LS}(I-M(1,.),V_{1},0)$, 
\end{center}
On the other hand, we have that
\begin{center}
$deg_{LS}(I-M(0,.),V_{0},0) = deg_{LS}(I-(P+QN_{f}+HP) ,V_{0},0)$.
\end{center}
But the range of the mapping
\begin{center}
$u \mapsto P(u)+Q(N_{f}(u))+H(P(u))$
\end{center}
is contained in the subspace  of related functions, isomorphic to $\mathbb{R}^{2}$. Using  homotopy invariance and reduction properties  of  Leray-Schauder degree \cite{man7}, we obtain
\begin{align*}
& deg_{LS}(I-(P+QN_{f}+HP) ,V_{0},0)\\
& = deg_{B}\left(I-(P+QN_{f}+HP)\left|_{\overline{\Delta}}\right. ,\Delta, 0\right)\\
&=deg_{B}(G,\Delta,0)\neq 0.
\end{align*}
Then, $deg_{LS}(I-M(1,.),V_{1},0) \neq 0$. This implies that, there exists $u \in V_{1}$ such that $ M(1,u)=u$, which is a solution for (\ref{diri1}). 
\end{proof}                                                                                         
\begin{remark}\label{obs}
Using the family of boundary value problems 
\begin{equation}\label{equa9}
\left\{\begin{array}{lll}
(\varphi(u' ))'   = \lambda N_{f}(u)+(1-\lambda)Q(N_{f}(u)) & & \\
u(T) = u'(0) = u'(T) 
\end{array}\right.
\end{equation} 
which gives the completely continuous homotopy $\widetilde{M}$ defined  by 
\begin{center}
 $\widetilde{M}(\lambda,u)=  S(u)+Q(N_{f}(u)) + K(\varphi^{-1}\left[\lambda H( N_f (u)-Q(N_{f}(u)))+\varphi(S(u))\right])$,
\end{center}
and similar a priori bounds as in the Lemma \ref{santos}, it is not difficult to see  that (\ref{equa9}) has a solution for $\lambda=1$.
\end{remark}

Let us give now an application of Theorem \ref{prin1}.
\begin{example}
Consider the boundary value problem
\begin{equation}\label{exemplo1.2}
\left\{\begin{array}{lll}
\left(\frac{u'}{\sqrt{1+u'^{2}}} \right)' = e^{4u'}-e & & \\
 u(0)=u'(0)=u'(T). 
\end{array}\right.
\end{equation}

Let $u\in C^{1}$, $M_{1}=0$ \ and    $M_{2}=\frac{1}{2}$. If we suppose that $u'_{m}\geq M_{2}$ and $u'_{M}\leq M_{1}$, then
\begin{center}
$\int_0 ^T (e^{4u'(t)}-e)dt\geq (e^{2}-e )T>0, \ \ \ \ \  \int_0 ^T (e^{4u'(t)}-e)dt\leq (1-e)T<0$.
\end{center}

Let $c(t)=-3$ for all $t\in \left[0, T\right]$, and let $L=\frac{1}{\sqrt{5}}$. If  $L+6T< \kappa=0,9<1$ and  $\rho > r(2+T)=\frac{L+6T}{\sqrt{1-(L+6T)^{2}}}(2+T)$, then the equation 
\begin{align*}
G(x,y)&=\left(-\frac{1}{T}\int_{0}^{T} f(t,x+yt,y)dt, y-x \right)=(0,0)\\
&= \left(-\frac{1}{T}\int_{0}^{T}(e^{4y}-e)dt, y-x \right)=(0,0)\\ 
&= \left(-e^{4y}+e, y-x \right)=(0,0)
\end{align*} 
has no solution on  $\partial \Delta$, and hence  the Brouwer degree $deg_{B}(G,\Delta,(0,0))$ is well  defined. So, using the properties of  the Brouwer degree, we have  
\begin{center}
 $deg_{B}(G, \Delta, (0,0))= \displaystyle\sum_{(x,y) \in G^{-1}(0,0)}$sgn$J_{G}(x,y)\neq 0$,
\end{center}
where $(0,0)$ is a regular value of $G$ and $J_{G}(x,y)=$det$ G'(x,y)$ is the Jacobian of $G$ at $(x,y)$. Therefore, the problem (\ref{exemplo1.2}) has at least one solution.  
\end{example} 
\section{Existence results for problems with bounded homeomorphisms}
\label{S:2}
In this section we  study the existence of at least one solution for nonlinear problems of the form
\begin{equation}\label{misto1}
\left\{\begin{array}{lll}
(\varphi(u' ))'   = f(t,u,u') & & \\
u(T)=u(0)=u'(T), 
\end{array}\right.
\end{equation}
where  $\varphi: \mathbb{R}  \rightarrow (-a,a) $ is  a homeomorphism, $\varphi(0)=0$ and $f:\left[0, T\right]\times \mathbb{R} \times \mathbb{R}\rightarrow \mathbb{R} $ is a continuous function such that
\begin{equation}\label{colo}
\left|f(t,x,y)\right| \leq c <\frac{a}{2T} \ \ \  for \ \  all  \ \ (t,x,y)  \in \left[0, T\right]\times \mathbb{R} \times \mathbb{R}.
\end{equation}
Now, using  Lemma \ref{lebema} and (\ref{colo}) we introduce the operator $M_{1}:C^{1}\rightarrow C^{1}$ defined by
\begin{center}
$M_{1}(u)=\varphi^{-1}(-Q_{\varphi}(K(N_{f}(u)))) + H\left(\varphi^{-1} \left[K(N_{f}(u))-Q_{\varphi}(K(N_{f}(u)))\right]\right)$.
\end{center}
The following results are taken from \cite{dallos}.
\begin{lemma}\label{dallos11}
If  $u \in C^{1}$ is such that $u= M_{1}(u)$, then $u$ is a solution of  (\ref{misto1}). 
\end{lemma}

\begin{lemma}\label{parecido}
The operator   $M_{1}:C^{1} \rightarrow C^{1}$ is completely continuous.
\end{lemma}

In order to apply Leray-Schauder degree to the  fixed point operator $M_{1}$, we introduce, for  $ \lambda \in [0,1]$, the family of boundary value problems
\begin{equation}\label{misto2}
\left\{\begin{array}{lll}
(\varphi(u' ))'  = \lambda f(t,u,u') & & \\
u(T)=u(0)=u'(T). 
\end{array}\right.
\end{equation}
Notice that (\ref{misto2}) coincide with (\ref{misto1}) for $\lambda =1$. For  each $ \lambda \in [0,1]$, we can define on $C^{1}$ the operator $M(\lambda,\cdot)$, where $M$  is defined on $[0,1]\times C^{1}$  by
\begin{equation}\label{colo2}
M(\lambda,u)=\varphi^{-1}(-Q_{\varphi}(\lambda K(N_{f}(u)))) + H\left(\varphi^{-1} \left[ \lambda K(N_{f}(u))-Q_{\varphi}( \lambda K(N_{f}(u)))\right]\right).
\end{equation}
Using the Arzel\`a-Ascoli  theorem it is not difficult to see  that $M$ is completely continuous.

\begin{lemma}\label{dallos11}
If $(\lambda,u)\in [0,1]\times C^{1}$ is such that $u=M(\lambda,u)$, then $u$ is a a solution of (\ref{misto2}).
\end {lemma}
\begin{proof} 
Let $(\lambda,u)\in [0,1]\times C^{1}$ be such that $ u=M(\lambda,u)$. Then
\begin{equation}\label{ptof}
u(t)=\varphi^{-1}(-Q_{\varphi}( \lambda K(N_{f}(u))))+ H\left( \varphi^{-1}\left[ \lambda K(N_{f}(u))-Q_{\varphi}( \lambda K(N_{f}(u)))\right]\right)(t)
\end{equation} 
for all $t\in\left[0, T\right]$. Since  $\int_0^T \varphi^{-1}\left[\lambda K(N_{f}(u))(t)-Q_{\varphi}( \lambda K(N_{f}(u)))\right]dt=0$, therefore, we have that  $u(0)=u(T)$. Differentiating  (\ref{ptof}), we obtain that 
\begin{center}
$u'(t)=\varphi^{-1}\left[ \lambda K(N_{f}(u))-Q_{\varphi}( \lambda K(N_{f}(u)))\right](t)$. 
\end{center}
 In particular,
\begin{center}
$u'(T)=\varphi^{-1}(0-Q_{\varphi}( \lambda K(N_{f}(u))))= \varphi^{-1}(-Q_{\varphi}( \lambda K(N_{f}(u))))=u(0)$.
\end{center}
Applying $\varphi$ to both members and  differentiating again, we deduce that
\begin{center}
$(\varphi(u'(t)))' = \lambda N_f (u)(t)$, \ \  $u(0)=u(T)$, \ $u(0)=u'(T)$
\end{center}
 for all $t\in\left[0, T\right]$. This completes the proof.
\end{proof} 

\begin{remark} 
Note that the reciprocal of Lemma \ref{dallos11} is  true because we can  guarantee that $\left\|\lambda K(N_{f}(u)\right\|_{\infty} <a/2$    for every solution   $u$ of (\ref{misto2}).
\end{remark}

Now we show the existence of at least one solution  for problem (\ref{misto1}) by means of Leray-Schauder degree. This result is inspired on works by Bereanu and Mawhin \cite{ma}.

\begin{theorem}\label{fot}
Let $f:[0,T]\times \mathbb{R}\times \mathbb{R} \longrightarrow \mathbb{R}$ be continuous. If $f$ satisfies the condition (\ref{colo}), then  the problem (\ref{misto1})  has at least one solution.
\end {theorem} 

\begin{proof}
Let $(\lambda,u)\in [0,1]\times C^{1}$ be such that $ u=M(\lambda,u)$. Then
\begin{center}
$u(t)=\varphi^{-1}(-Q_{\varphi}( \lambda K(N_{f}(u))))+ H\left( \varphi^{-1}\left[ \lambda K(N_{f}(u))-Q_{\varphi}( \lambda K(N_{f}(u)))\right]\right)(t)$
\end{center}
for all $t\in\left[0, T\right]$. Differentiating, we obtain that
\begin{align*}
u'(t)&=\varphi^{-1}\left[\lambda K(N_f(u))- Q_{\varphi}(\lambda K(N_f(u)))\right](t)\\
&= \varphi^{-1}\left[\lambda K(N_f(u))(t)- Q_{\varphi}(\lambda K(N_f(u)))\right].
\end{align*}
Applying   $\varphi$  to both of its members we have that 
\begin{center}
 $ \varphi(u'(t))= \lambda K(N_f(u))(t)- Q_{\varphi}(\lambda K(N_f(u)))  \ \ \ (t\in [0, T])$.
\end{center}
Using the fact that $f$ is bounded, we deduce the elementary inequality
\begin{center}
$\left\|\varphi(u') \right\|_{\infty}\leq 2cT<a$.
\end{center}
Hence,
\begin{center}
$\left\|u'\right\|_{\infty}\leq L$,
\end{center}
where  $L= \max\left\{\left|\varphi^{-1}(2cT)\right|,\ \left|\varphi^{-1}(-2cT)\right|\right\}$.
Because  $u\in C^{1}$ is such that  $u(0)=u'(T)$ we have that
\begin{center}
$\left|u(t)\right|\leq \left|u(0)\right|+ \int_ 0 ^T  |u'(s)| ds \leq L +LT   \  \  \  (t\in[0,T] ) $,
\end{center}
and hence 
\begin{center}
$\left\|u\right\|_{1} =\left\|u\right\|_{\infty}+ \left\|u'\right\|_{\infty} \leq L+LT +L=L(2+T)$.
\end{center}

Let $M$ be operator given by (\ref{colo2}) and let $\rho>L(2+T)$. Using the homotopy invariance of the  Leray-Schauder degree, we have
\begin{center}
$deg_{LS}(I-M(0,\cdot),B_{\rho}(0),0)= deg_{LS}(I-M(1,\cdot),B_{\rho}(0),0)$.
\end{center}
On the other hand, we have that
\begin{center}
$deg_{LS}(I-M(0,\cdot),B_{\rho}(0),0) = deg_{LS}(I ,B_{\rho}(0),0)=1$.
\end{center}
Then, from the existence property of Leray-Schauder degree \cite{man7}, there exists $u\in B_{\rho}(0)$ such that $u=M(1,u)$, which is a solution for (\ref{misto1}).
\end{proof} 

\begin{example}
Consider the boundary value problem
\begin{equation}\label{exemplo1.3}
\left\{\begin{array}{lll}
\left(\frac{u'}{\sqrt{1+u'^{2}}} \right)' = \beta cos u & & \\
 u(0)=u(T)=u'(T). 
\end{array}\right.
\end{equation}
So, we can choose $\beta<\frac{1}{2T}$ to see Theorem \ref{fot} holds and so the problem \ref{exemplo1.3} has at least one solution.
\end{example}



\section*{Acknowledgements} 
This research was supported by  CAPES and CNPq/Brazil. The author would like to thank to Dr. Pierluigi Benevieri for his kind advice and for the constructive revision of this paper.

\bibliographystyle{plain}

\renewcommand\bibname{References Bibliogr\'aficas}

\end{document}